\documentclass[12pt]{amsart}
\usepackage{amsmath,amsthm,amssymb,bbm,dsfont}
\usepackage{graphicx,mathrsfs}
\usepackage{subfig}
\usepackage{cite,float}

\numberwithin{equation}{section}
\newtheorem{theorem}{Theorem}[section]
\newtheorem{lemma}[theorem]{Lemma}

\newtheorem{proposition}[theorem]{Proposition}
\theoremstyle{remark}
\newtheorem{remark}[theorem]{Remark}

\newcommand{\G}{\mathbb{G}_2}
\newcommand{\C}{\mathbb{C}^n}
\newcommand{\D}{\mathbb{D}}
\newcommand{\ld}{\lambda}
\newcommand{\Gs}{\partial_s G(2\lambda, \lambda^2)}
\newcommand{\Gp}{\partial_pG(2\lambda, \lambda^2)}
\newcommand{\fGs}{\partial_sG(0,\ld)}

\makeatletter
\@namedef{subjclassname@2010}{%
	\textup{2010} Mathematics Subject Classification}
\makeatother
\begin{document}
	
	\title{Lempert left inverses of the Royal and Flat geodesics of the symmetrized bidisc}
	
	\author[N. Gupta]{Naveen Gupta}
	\address{Department of Mathematics\\
		Netaji Subhas University of Technology\\
		Delhi-110078\\
		India}
	\email{naveen.gupta@nsut.ac.in, ssguptanaveen@gmail.com}

	\begin{abstract}
		In this note, we give complete description of the Lempert left inverses of the royal 
		and flat geodesics of the symmetrized bidisc. It is further proved that the 
		collection of all such Lempert left inverses to the royal and flat geodesics 
		is path connected. This, in particular, answers a question asked by Zwonek 
		\cite[~Remark 4.4, 4.5]{zwonek}.
	\end{abstract}
	\maketitle
	
	\section{Introduction}

A fundamental theorem of Lempert states that the Carath\'eodory pseudodistance and 
the Lempert function agrees when the domain $D\subseteq \C,\, n\geq 2$ is strongly 
linearly convex or convex \cite{lempert, lempert1, lempert2}. A taut domain in $\C$ for which the Lempert theorem holds 
is said to be a Lempert domain. Equivalently, if every pair of distinct points in the 
domain $D$ lies in the image of a holomorphic disc that admits holomorphic left inverse. 
Such a disc is called complex geodesic and the corresponding holomorphic left inverse 
is said to be its left inverse. 
A left inverse whose fibres are intersection of the domain with a complex 
affine hyperplane is called Lempert left inverse (or a good left 
inverse) of the complex geodesic. 

Zwonek in his recent paper \cite{zwonek} presented a construction that produces a Lempert 
left inverse of an arbitrary complex geodesic of a bounded convex domain. For the 
symmetrized bidisc $G_2$ (observe that it is non-convex), he considered the royal and the 
normalized flat geodesics. For the royal geodesic, he presented a family of all possible 
left inverses \cite[~Proposition 4.3 ]{zwonek} and a family of Lempert left inverses 
\cite[~Equation 4.9, 4.16]{zwonek}.  In case of normalized flat geodesic, he presented 
a family of Lempert left inverses given in \cite[~Section 4.4]{zwonek} and leaves 
the description of all Lempert left inverses and that of its connectedness as a question.

We prove that a map $G:\G\to \mathbb{D}$ is a Lempert left inverse of the royal geodesic 
$f_r$ of $G_2$ if and only if $G=H_{\varphi}$ for some $\varphi\in \mathcal S$ (see 
Theorem \ref{thm:royal_H} and Theorem \ref{thm:royal_every}). For the case of flat 
geodesic, we prove that  $G$ is Lempert left inverse of the flat geodesic $f_0$ if and only 
if $G=G_c$ for some $c\in \mathbb{C}$ with $|c|\leq 1/2$ (see Proposition \ref{prop:flat_Gc}
 and Theorem \ref{thm:flat_every}). The case of a general flat geodesic is discussed using 
 the automorphism of $G_2$. As a result, we prove that the collection of all 
 Lempert left inverses of royal and flat geodesics are path connected, answering 
 a question (\cite[~Remark 4.4, 4.5]{zwonek}) posed by Zwonek in affirmative.


\section{Lempert left inverses of the royal geodesic}

The symmetrized bidisc, $\mathbb{G}_2$ is defined as 
$$\mathbb{G}_2:=\{(z+w, zw):\,z,w\in \mathbb{D}\},$$ where $\mathbb{D}\subseteq \mathbb{C}$ 
is the unit disc. It is a Lempert domain that is not biholomorphic to a convex domain. 
In fact, it can not even 
be exhausted by domains that are biholomorphic to a  convex domain \cite{agler, costara, 
	armen}. We denote the co-ordinates of $\mathbb{G}_2$ by $(s,p).$

\begin{lemma}\label{lem:aux}
	For any $(s,p)=(z+w,zw)\in \mathbb{G}_2$, and for each $\zeta$ with 
	$|\zeta|=1$, we have 
	$$|2p-s\zeta|<|2\zeta-s|.$$
\end{lemma}

\begin{proof}
	Consider 
	\begin{align*}
		|s-\bar s p|&= |z(1-|w|^2)+w(1-|z|^2)|\\
		&\leq |z|(1-|w|^2)+|w|(1-|z|^2)\\
		&\leq (|z|+|w|)(1-|zw|)\\
		&< (1+|zw|)(1-|zw|)\\
		&=(1-|zw|^2).
	\end{align*}

For any $\zeta$ with $|\zeta|=1$, consider 
\begin{align*}
	|2\zeta-s|^2-|2p-s\zeta|^2&=4\left((1-|p|^2)-Re ((s-\bar{s}p)\bar{\zeta})\right)\\
	&\geq 4\left((1-|p|^2)-|s-\bar{s}p|\right).
\end{align*}
Now the desired inequality follows from the previous inequality.
\end{proof}
We denote $\dfrac{\partial G}{\partial s}$ by $\partial_s G$ and  
$\dfrac{\partial G}{\partial p}$ by $\partial_p G$. $\mathcal{S}$ denotes 
the collection of all holomorphic functions $\varphi:\mathbb{D}\to 
\mathbb{C}$ with $|\varphi(\ld)|\leq 1$ for each $\ld\in \D$.
 
The royal geodesic $f_r:\D\to \G$ of $\G$ is given as $f_r(\ld):=(\ld,\ld^2)$. 
For any $\beta\in \mathbb{C}$ with $|\beta|<1$, the flat geodesic of $\G$ 
denoted by $f_{\beta}:\D\to \G$ is defined as
 $f_{\beta}:=(\beta+\bar{\beta}\ld, \ld)$. The flat geodesic corresponding 
 to $\beta=0$ denoted by $f_0$ is called normalized flat geodesic. We begin with the 
 description of Lempert left inverse of the royal geodesic and the case of flat 
 geodesic in discussed in the next section.

\begin{lemma}\label{lem:main}
	Let $G:\G\to \D$ be any left inverse of $f_r$. Then 
	$$\left|\Gp\right|\leq \left|1-\lambda \Gp\right|.$$
	In particular, the map ${\varphi}_G:\D\to \mathbb{C}$ defined as 
	$${\varphi}_G(\lambda):=\dfrac{\Gp}{\lambda\Gp-1}$$ 
is holomorphic on $\D$ and it satisfies $|\varphi_G(\lambda)|\leq 1$.
 Moreover, 	$\partial_sG(f_r(\lambda))=\dfrac{1+\lambda {\varphi}_G(\lambda)}
 {2(1-\ld {\varphi}_G(\ld))}$ and 
$\partial_pG(f_r(\lambda))=\dfrac{{\varphi}_G(\lambda)}{\ld {\varphi}_G(\ld)-1}.$
\end{lemma}
\begin{proof}
	Using \cite[~Theorem 2.1(3)]{agler}, we have $\G=\{(s,p):
	|p|<1, |s-\bar{s}p|<1-|p|^2\}.$ 
	For $\lambda\in \C$ fixed, set $\beta=\dfrac{2\ld}{1+|\ld|^2}$. Observe that 
	$|\beta|<1.$ Consider the map $\ell:\D\to \G$ defined as 
	$\ell(\mu):=(\beta+\bar{\beta}\mu, \mu)$. It is easy to check that 
$$|\beta+\bar{\beta}\mu-(\beta+\bar{\beta} \mu)\mu|=|\beta||1-|\mu|^2|<(1-|\mu|^2)$$ and 
therefore $\ell(\D)\subseteq \G$. Now, consider $h:\D\to \D$ defined as
 $h(\mu):=G(\ell(\mu))$. 
It follows that $h(\lambda^2)=Id(\lambda)$. Therefore, using  Schwarz-Pick Lemma it
 follows that 
$|h'(\lambda^2)|^2\leq \dfrac{1-|h(\lambda^2|^2)}{1-|\lambda^2|^2}=
\dfrac{1}{1+|\lambda|^2}.$ 
 Since $h(\mu)=G(\ell(\mu))$, therefore $h'(\mu)=G'(\ell(\mu))\ell'(\mu)
 =\bar{\beta}\partial_s G(\ell(\mu))+\partial_p G(\ell(\mu)).$ Thus, 
$h'(\lambda^2)=\dfrac{2\bar{\ld}}{1+|\ld|^2}\partial_s G(2\ld, \ld^2)
+\partial_p G(2\ld, \ld^2)$. This upon using the inequality obtained 
earlier yields 
\begin{equation}\label{lem:main_eq}
\left|2\bar{\ld}\partial_s G(2\ld, \ld^2)+(1+|\ld|^2)
\partial_p G(2\ld, \ld^2)\right|\leq 1.
\end{equation} 
Using that $2\partial_s G(2\ld, \ld^2)+
2\ld \partial_p G(2\ld, \ld^2)=1$, then 
squaring and further simplifying, we obtain 
$\left|\Gp\right|\leq \left|1-\lambda \Gp\right|$ and consequenctly it follows that 
${\varphi}_G$ is holomorphic with absolute value less than equal to 1. 
 
 Also, since $G$ is left inverse of $f_r$, therefore $G\circ f_r=\text{Id}_{\mathbb{D}}$. 
 	So, using chain rule, we get $2\Gs+2\ld\Gp=1$ for all $\ld \in \D$, and the desired expression of the 
 	partial derivatives of $G$ along $f_r(\ld)$ follows.
 
\end{proof}

\begin{proposition}\label{prop}
	For each function ${\varphi}$ in $\mathcal{S}$ and $(s,p)\in \G$, the equation
	\begin{equation}\label{eq:prop}
2\lambda-s+(2p-s\lambda){\varphi}(\ld)=0		
	\end{equation} has exactly one solution $\ld \in \D$, and this solution is simple. 
Denote this solution by $H_{\varphi}(s,p)$. Then the map $(s,p)\to H_{\varphi}(s,p)$ defines a holomorphic function 
on $\G.$
\end{proposition}

\begin{proof}
	Fix any ${\varphi}\in \mathcal{S},\, (s,p)\in \G$ and define 
	$$\Phi_{s,p}(\lambda):=2\ld-s+(2p-s\ld){\varphi}(\ld).$$ It is easy to see using Lemma
	 \ref{lem:aux} that the map $\zeta\to |2\zeta-s|-|2p-s\zeta|$ is positive on the unit
	  circle. Furthermore, using compactness of the unit circle and continuity of this function,
	   there exists $\delta_0>0$ such that it is positive on 
	   $\delta_0<|\zeta|<1$. Choose $\rho$ with 
	   $\max\{\delta_0,|s|/2\}
	   <\rho<1.$
	   
	   On the circle $|\zeta|=\rho$, we have
	    $$|(2p-s\zeta){\varphi}(\zeta)|\leq |2p-s\zeta|<|2\zeta-s|$$ and therefore Roche's theorem tells 
	    us that $\Phi_{s,p}$ has the same number of zeros (counted according to multiplicty) as the 
	    function $2\zeta-s$ inside the circle $|\zeta|=\rho$. As $|s/2|<\rho$, the 
	    function $2\zeta-s$ has exactly one zero inside $|\zeta|=\rho$ and thus 
	    $\Phi_{s,p}$ has exactly one simple zero inside $|\zeta|=\rho$ and hence inside 
	    $\D$ as $\rho <1$ was arbitrary. We denote this zero by $H_{\varphi}(s,p)$.
	    
	    We claim that the function $H_{\varphi}:\G\to \D$ defined by $(s,p)\to H_{\varphi}(s,p)$ is 
	    holomorphic. 
	    
	   For this, we consider the function $\Phi$ defined as 
	   $(\phi(s,p,\ld)):=2\ld-s+(2p-s\ld){\varphi}(\ld).$ It is holomorphic in each variable and 
	   therefore it is holomorphic. For $(s_0,p_0)\in \G$ and $H_{\varphi}(s_0,p_0)=\ld_0$,
	    we have 
	   $\Phi(s_0,p_0,\ld_0)=0$. Also, $\partial_{\ld}\Phi(s_0,p_0,\ld_0)\neq 0$, since 
	   the zero $\ld_0$ is simple. Therfore, using implicity function theorem, there exists 
	   a neighborhood $U\subseteq \G$ of $(s_0,p_0)$, a neighborhood 
	    $V\subseteq \D$ of $\ld_0$ and a holomorphic function $h:U\to V$ such that 
$$h(s_0,p_0)=\ld_0 \, \text{and} \, \Phi(s,p,h(s,p))=0\,\, \text{for all}\,\,(s,p)\in U.$$
Thus, for each $(s,p)\in U$, the point $h(s,p)\in \D$ satisfies the equation
 $\Phi_{s,p}(h(s,p))=0$. 
Since the zeros are simple for any $(s,p)\in \G$, therefore we get $h(s,p)=H_{\varphi}(s,p)$ on 
$U$ and hence $H_{\varphi}$ is holomorphic, as the point $(s_0,p_0)$ was arbitrary.
	 \end{proof}

\begin{theorem}\label{thm:royal_H}
The function $H_{\varphi}$ for any ${\varphi}\in \mathcal{S}$ is Lempert left inverse of $f_r$ with its 
fibres given by $$H_{\varphi}^{-1}(\ld):=\{(s,p)\in \G:2\lambda-s+(2p-s\lambda){\varphi}(\ld)=0	\}.$$	
\end{theorem}
\begin{proof}
	For any ${\varphi}\in \mathcal{S}$, and $(s,p)\in \G$ 
	the equation 
	
	$2H_{\varphi}(s,p)-s+(2p-sH_{\varphi}(s,p)){\varphi}(H_{\varphi}(s,p))=0$
	 holds. Thus, 
	$$H_{\varphi}(s,p)=\dfrac{s-2p{\varphi}(H_{\varphi}(s,p))}{2-s{\varphi}(H_{\varphi}(s,p))}$$ and therefore, 
	it follows that $H_{\varphi}$ is a Lempert left inverse of $f_r$ 
	by taking $h={\varphi}\circ H_{\varphi}$ in \cite[~Proposition 4.3]{zwonek}. The
	 expression of the fibres follows from the definition.
\end{proof}
\begin{theorem}\label{thm:royal_every}
Every Lempert left inverse of $f_r$	is of the form $H_{\varphi}$ for some unique 
${\varphi}\in \mathcal{S}$.
\end{theorem}
\begin{proof}
Let $G$ be a Lempert left inverse of $f_r$. Denote $\varphi:=\varphi_G$. Using Lemma 
\ref{lem:main}, we know  	${\varphi}_G(\lambda):=\dfrac{\Gp}{\lambda\Gp-1}$. We claim 
that $G=H_{\varphi}$. We prove $G^{-1}(\ld)=H^{-1}_{\varphi}(\ld)$. Since $G$ is Lempert 
left inverse of $f_r$, therefore 
$$G^{-1}(\ld)=\G\cap \left(f_r(\ld)+\text{Ker}\, G'(f_r(\ld))\right).$$

It is easy to check that a poin $(s,p)\in \G$ belongs to $f_r(\ld)+\text{Ker}\, G'(f_r(\ld))$
if and only if $G'(f_r(\ld))\bullet (s-2\ld, p-\ld^2)=0$. Which holds if and only if 
$$\Gs(s-2\ld)+\Gp(p-\ld^2)=0.$$ Using the expression of $\varphi$, this holds if and only if
$$2\lambda-s+(2p-s\lambda){\varphi}(\ld)=0.$$ Thus, we obtain $G=H_{\varphi}$ using 
Theorem \ref{thm:royal_H}. 
Uniqueness follows by differentiating $H_{\varphi}$ along the royal geodesics and 
using the expression of $q$ in terms of the partial derivatives of $H_{\varphi}$. 
\end{proof}

\begin{remark}
The Lempert left inverse  $\Psi_w$ for $|w|=1$ of $f_r$ presented by Zwonek in \cite[~Equation 4.9]{zwonek}
equals $H_\varphi$, for $\varphi\equiv \omega$ and $\Phi$ presented in 
\cite[Equation 4.16]{zwonek} equals 
$H_{\varphi}$ for
 $\varphi=\text{Id}_\mathbb{D}$. 
\end{remark}
\begin{remark}
	In the notation of \cite[~Proposition 4.3]{zwonek}, the mapping 
	$\Psi:\G\to \mathbb{D}$ is the Lempert left inverse of the complex geodesic $f_r$ if 
	and only if $$\Psi(s,p)=\dfrac{s-2ph(s,p)}{2-sh(s,p)},$$ where $h:\G\to \overline{D}$ is 
	of the form $h=\varphi\circ H_{\varphi}$ for some $\varphi \in \mathcal{S}.$
\end{remark}
\begin{theorem}\label{thm:connected_royal}
	The collection $\mathcal{L}_r$ or all Lempert left inverses of the royal geodesic 
	$f_r$ of the symmetrized bidisc is path connected.
\end{theorem}
\begin{proof}
	Let $G_0,\, G_1\in \mathcal{L}_r$. Thus Theorem \ref{thm:royal_every}, tells us that 
	$G_0=H_{\varphi_0},\, G_1=H_{\varphi_1}$ for $\varphi_0, \, \varphi_1 \in 
	\mathcal{S}$. For $t\in[0,1]$, take $\varphi_t:=(1-t)\varphi_0+t\varphi_1$. It is 
	easy to see that $\varphi_t\in \mathcal{S}$ for all $t\in[0,1].$ Thus, Theorem 
	\ref{thm:royal_H}, implies that each of the $H_{\varphi_t}$ belong to 
	$\mathcal{L}_r$. This completes the proof.
\end{proof}
\section{Lempert left inverse to the flat geodesic}

\begin{lemma}\label{lem:flat_main}
	Let $G:\G \to \D$ be a left inverse of the flat geodesic $f_0$. 
	Then there exists a unique constant $c\in \mathbb{C}$ with $|c|\leq 1/2$ such that 
	$$\partial_sG(0,\ld)=c-\bar{c}\ld\,\ \ \ \ \ \ \  \partial_pG(0,\ld)=1.$$ 
	In particular, $c=\partial_sG(0,0).$
\end{lemma} 
\begin{proof}
Consider the  complex valued functions $f$ and $F$ defined on $\D$ and $\D^2$ respectively as 
  $f(\ld):=\fGs$ and 	$F(z,w):=G(z+w,zw)$. Since $G$ is left inverse of $f_0$, it 
  is easy to see that for any $\mu \in \D$, 
  $F(\mu,-\mu)=-\mu^2$, $\partial_zF(\mu, -\mu)=f(-\mu^2)-\mu$ and 
  $\partial_wF(\mu, -\mu)=f(-\mu^2)+\mu$. Now, using the \cite[~Proposition 1.1]{knese} for
   $F$ at $(\mu, -\mu)$ and denoting $-\mu^2:=\zeta$, we get 
   $$|f(\zeta)-\mu|+|f(\zeta)+\mu|\leq 1+|\mu|^2.$$ Squaring and simplifying, we get 
   $$|f(\zeta)^2+\zeta|\leq \dfrac{|1+|\zeta|^2}{2}-|f(\zeta)|^2.$$
   Squaring again and simplifying, we get 
   \begin{equation}\label{eq:lem_flat}
   |f(\zeta)-\zeta\bar{f(\zeta)}|\leq \dfrac{1-|\zeta|^2}{2}.
   \end{equation}
For $|z|<r<1$, consider the integral 
$$\left|\dfrac{1}{2\pi i}\int_{|\zeta|=r}\dfrac{f(\zeta)+\zeta \bar{f(\zeta)}}{\zeta - \ld}
d\zeta \right|\leq \dfrac{1}{2\pi}\dfrac{r(1-r^2)}{2(r-|\ld|)}.$$
The expression on the left goes to $0$ as $r\uparrow 1$ (using Equation \ref{eq:lem_flat}). 
On the other hand, direct computations show that the integral equals 
$f(\ld)+\ld\bar{f(0)}+r^2\bar{f'(0)}$ and it tends to $f(\ld)+\ld f(0)+\bar{f'(0)}$ as
 $r\uparrow 1$. Therefore, we get $f(\ld)+\ld f(0)+\bar{f'(0)}=0$. In particular, using it 
 at $\ld=0$ gives $\bar{f'(0)}=-f(0)$ and hence the result of the lemma concludes.
\end{proof}

\begin{proposition}\label{prop:flat_Gc}
	For every $c\in \mathbb{C}$ with $|c|\leq 1/2$, the function $$G_c(s,p):=
	\dfrac{p+cs}{1+\bar{c}s}$$ is a Lempert left inverse of $f_0$ and its fibres 
	are given by $G_c^{-1}(\ld):=\{(s,p)\in \G:(c-\bar{c}\ld)s+p=\ld\}$. Moreover, 
	$\partial_sG_c(0,\ld)=c-\bar{c}\ld$ and $\partial_pG_c(0,\ld)=1.$
\end{proposition}
\begin{proof}
Fix $c\in \mathbb{C}$, it is easy to see that 
$$|1+\bar{c}s|^2-|p+cs|^2\geq (1-|p|^2)-2|c||s-p\bar{s},|$$  and that the 
expression on the left is positive for $|c|\leq 1/2$ (using \cite[~Theorem 2.1(3)]{agler}). Therefore, 
the image of $G_c$ lies within $\D$. Also, since $G_c(0,\ld)=\ld$, it follows that 
it is left inverse of $f_0$. Also, it is straightforward to 
verify that $\partial_sG_c(0,\ld)=c-\bar{c}\ld$ and $\partial_pG_c(0,\ld)=1.$ 
In order to prove that it is actually a Lempert left inverse 
of $f_0$, we claim that 
$$G_c^{-1}(\ld)=\G\cap \left(f_0(\ld)+ \text{ker}G'_c(f_0(\ld))\right).$$ 
In fact, 
\begin{align*}
(s,p)\in G_c^{-1}(\ld) &\iff (c-\bar{c}\ld)s+p=\ld\\
&\iff \left(\partial_sG_c(f_0(\ld)), \partial_pG_c(f_0(\ld))\right)\bullet ((s,p)-f_0(\ld))=0\\
&\iff (s,p)-f_0(\ld) \in \text{Ker} G'(f_0(\ld)).
\end{align*}

Thus, each fibre of $G_c$ is intersection of an affine hyperplane with $\G$ and hence 
$G_c$ is a Lempert left invrese of $f_0$.
\end{proof}

\begin{remark}
Observe that $G_c= \Psi_{-2c}$, where $\Psi_w:\G\to \mathbb{D}$ is as defined in 
\cite[~Equation 4.9]{zwonek}. Zwonek  proved that this family consists of Lempert 
left inverses of $f_0$. We prove that every Lempert left inverse of $f_0$ are of this form 
in the following theorem.
  \end{remark}

\begin{theorem}\label{thm:flat_every}
Every Lempert left inverse $G\to \C$ of $f_0$ is of the form $G_c$ for a unique 
$c\in \mathbb{C}$ with $|c|\leq 1/2.$

\end{theorem}
\begin{proof}
	Since $G$ is left inverse of $f_0$, using Lemma \ref{lem:flat_main} there is 
a unique	$c\in \mathbb{C} \, (c=\partial_sG(0,0))$ with $|c|\leq 1/2$ and that 	$\partial_sG_c(0,\ld)=
	c-\bar{c}\ld$ and $\partial_pG_c(0,\ld)=1.$ W claim to show that $G=G_c$. Observe that 
	it is enough to show that $G^{-1}(\ld)=G_c^{-1}(\ld)$ for each $\ld\in \D$. In fact, 
	$(s,p)\in G^{-1}(\ld)=\G\cap \{f_0(\ld)+\text{Ker}G'(f_0(\ld))\}$ if and only if 
	$G'(f_0(\ld))\bullet (s,p-\ld)=0$. Which upon using the expression of partial 
	derivatives of G stated above simplifies to $(c-\bar{c}\ld)s+p=\ld$. This 
	using the expression of fibres of $G_c^{-1}$ mentioned in Proposition 
	\ref{prop:flat_Gc} completes the proof.
\end{proof}

\begin{theorem}\label{thm:connected_flat}
	The collection $\mathcal{L}_f$ of all Lempert left inverses of the flat geodesic 
	$f_{0}$ of the symmetrized bidisc is path connected.
\end{theorem}

\begin{proof}
	Proof follows on the similar lines as in the case of royal geodesics case using 
	Proposition \ref{prop:flat_Gc} and Theorem \ref{thm:flat_every}.
\end{proof}
For any $a \in \mathbb{D}$, we denote by $\phi_a$  the disc automorphism of $\mathbb{D}$ 
with  $\phi_a(a)=0$. We define $T:\G\to \G$ as 
$$T(s,p)=T(z+w,zw):=(\phi_a(z)+\phi_a(w),\phi_a(z)\phi_a(w)).$$
\begin{proposition}
For a fixed $\beta \in  \mathbb{D}$, consider the flat geodesic $f_{\beta}$ of $\G$ given by 
$f_{\beta}(\ld):=(\beta+\bar{\beta}\ld, \ld)$. Then a map 
$K:\G\to \mathbb{D}$ is Lempert left inverse of $f_{\beta}$ if and only if 
$K=\phi_{a^2}^{-1}\circ G_c\circ T$,  for some unique $c$ with $|c|\leq 1/2$, 
where $G_c$	is as discussed in Proposition \ref{prop:flat_Gc} and $T$ is as defined above.
\end{proposition}
\begin{proof}
Let $K$ be a Lempert left inverse of $f_{\beta}$. Consider $L:\G \to \mathbb{D}$ 
defined as $L:=\phi_{a^2}\circ K\circ T^{-1}$. It follows  that
$T\circ f_{\beta}(\ld)=(0, \phi_{a^2}(\ld))=f_0\circ \phi_{a^2}(\ld)$ for 
any $\ld \in \mathbb{D}$. Therefore, $ T^{-1}\circ f_0=f_{\beta}\circ \phi_{a^2}^{-1}.$ 
This implies that 
\begin{align*}
L\circ f_0&=\phi_{a^2}\circ K\circ T^{-1}\circ f_0\\
&=\phi_{a^2}\circ K\circ 
f_{\beta}\circ \phi_{a^2}^{-1}\\
&=\phi_{a^2}\circ \text{Id}_{\mathbb{D}}\circ \phi_{a^2}^{-1}=\text{Id}_{\mathbb{D}}.
\end{align*} 
Also, $L^{-1}=T\circ K^{-1} \circ \phi_{a^2}^{-1}$. Since $T$ preserve affine hyperplane 
sections and $K$ being Lempert left inverse its fibres are  intersection of 
$\G$ with its affine complex hyperplanes, we obtain that $L$ is Lempert left inverse of $f_0$.
Hence,  using Theorem \ref{thm:flat_every}, $L=G_c$ for some $c\in \mathbb{C}$ with $|c|\leq 1/2$ and we obtain that  
$K=\phi_{a^2}^{-1}\circ G_c \circ T$. 

Conversely, assume that $K=\phi_{a^2}^{-1}\circ G_c\circ T$ for some $c\in \mathbb{D}$ with 
$|c|\leq 1/2$.  Since $T\circ f_{\beta}(\ld)=(0, \phi_{a^2}(\ld))$, therefore
$G_c\circ T\circ f_{\beta}(\ld)=\phi_{a^2}(\ld)$. Hence, 
$\phi_{a^2}^{-1}\circ G_c\circ T\circ f_{\beta}(\ld)=\text{Id}_{\mathbb{D}}(\ld)$. Also,
 for each $\ld \in \D$, 
$K^{-1}(\ld)=T^{-1}\circ G_c^{-1}\circ \phi_{a^2}$. Since $G_c$ is Lempert 
left inverse of $f_0$ and $T^{-1}$ preserves affine fibres, the proof concludes. 
\end{proof}

\begin{remark}
	The proof for path connectedness property of the collection $\mathcal{L}_{\beta}$ of all Lempert left 
	inverses of the geodesic $f_{\beta}$ follows on the similar lines as in Theorem 
	\ref{thm:connected_royal} and Theorem \ref{thm:connected_flat}.
\end{remark}



\medskip


\begin{thebibliography}{00}
	\smallskip
	
	
	\bibitem{agler} J. Agler, N. J. Young, The hyperbolic geometry of ths symmetrized 
	bidisc,  J. Geom. Anal., {\bf 14}, 2004, 375--403.
	
	\bibitem{costara} C. Costara, The symmetrized bidisc and Lempert's theorem, 
	Bull. London Math. Soc., {\bf 36}, 2004, 656--662.
	
	\bibitem{armen} A. Edigarian, A note on Costara's paper: The symmetrized bidisc and 
	Lempert's theorem, Ann. Polon. 
	Math., {\bf 83}, 2004, 189--191.
	
	\bibitem{knese} G. Knese, A Schwarz Lemma on the polydisk, Proc. Amer. Math. Soc., 
	{\bf 135}(9), 2007, 2759--2768.

\bibitem{lempert} L. Lempert, La metrique de Kobayashi et la repr\'esentation des 
domaines sur la boul\'e , Bull. Soc. Math. France {\bf 109}(4), 1981,  427--474.
\bibitem{lempert1}  L. Lempert, Holomorphic retracts and intrinsic metrics in convex 
domains, Anal. Math. 8, 1982, 257--261.
\bibitem{lempert2} L. Lempert, Intrinsic distances and holomorphic retracts, in Complex
analysis and applications’81 (Varna,
1981), 341--364, Publ. House Bulgar. Acad. Sci., Sofia, 1984.	

	
\bibitem{zwonek} W. Zwonek, On left inverses in the Lemepert theorem, Proc. Amer. Math. 
	Soc., 2027.
	


	
	
	
	
\end{thebibliography}
\end{document}